\documentclass[11pt,leqno]{amsart}
\usepackage{amssymb,amsfonts,amscd, color}
\usepackage{amsfonts,amsmath,amssymb,amscd,subfigure}
\usepackage{srcltx} 
\usepackage{latexsym} 
\usepackage{esint}
\usepackage{multicol}
\usepackage{mathrsfs}
\usepackage{algorithm}
\usepackage{bm}
\usepackage[plainpages=false,pdfpagelabels,hypertexnames=false,colorlinks=true,
pdfstartview=FitV,linkcolor=red,citecolor=blue,urlcolor=blue]{hyperref}

\usepackage{verbatim}

\def\nc{\newcommand}

\def\om{\omega}

 \def\Om{\Omega}

\nc\pa{\partial}

\nc\CC{\mathbb{C}}
\nc\RR{\mathbb{R}}
\nc\QQ{\mathbb{Q}}
\nc\ZZ{\mathbb{Z}}
\nc\NN{\mathbb{N}}

\nc\m[1]{\left| #1\right|}
\nc\norm[1]{\left\|#1\right\|}
\newtheorem{theorem}{Theorem}[section]
\newtheorem{lemma}[theorem]{Lemma}

\newtheorem{definition}[theorem]{Definition}
\newtheorem{remark}[theorem]{Remark}        
\numberwithin{equation}{section}

\begin{document}

\title[Quasilinear Riccati type equations]
{Quasilinear Riccati type equations with  oscillatory and singular data}

\author[Quoc-Hung Nguyen]
{Quoc-Hung Nguyen}
\address{ShanghaiTech University,
393 Middle Huaxia Road, Pudong,
Shanghai, 201210, China.}
\email{qhnguyen@shanghaitech.edu.cn}

\author[Nguyen Cong Phuc]
{Nguyen Cong Phuc}
\address{Department of Mathematics,
Louisiana State University,
303 Lockett Hall, Baton Rouge, LA 70803, USA.}
\email{pcnguyen@math.lsu.edu}


\begin{abstract} We characterize the existence of solutions to the quasilinear Riccati type equation
\begin{eqnarray*}
\left\{ \begin{array}{rcl}
-{\rm div}\,\mathcal{A}(x, \nabla u)&=& |\nabla u|^q + \sigma  \quad \text{in} ~\Omega, \\
u&=&0  \quad \text{on}~ \partial \Omega,
\end{array}\right.
\end{eqnarray*}
with a distributional or measure datum $\sigma$. Here ${\rm div}\,\mathcal{A}(x, \nabla u)$ is a quasilinear elliptic operator modeled after 
the $p$-Laplacian ($p>1$),  and $\Om$ is a bounded  domain whose boundary is sufficiently flat (in the sense of Reifenberg). For distributional data, we assume that $p>1$ and $q>p$. For measure data, we assume that they are compactly supported in $\Om$,  $p>\frac{3n-2}{2n-1}$, and $q$ is in the sub-linear range $p-1<q<1$. We also assume more regularity conditions  on $\mathcal{A}$ and on 
$\partial\Om$ in this case.
\end{abstract}

\maketitle

\section{Introduction and main results}

We  address in this note the question of existence for the quasilinear Riccati type equation
\begin{eqnarray}\label{Riccati11}
\left\{ \begin{array}{rcl}
-{\rm div}\,\mathcal{A}(x, \nabla u)&=& |\nabla u|^q + \sigma  \quad \text{in} ~\Omega, \\
u&=&0  \quad \text{on}~ \partial \Omega,
\end{array}\right.
\end{eqnarray}
where the datum $\sigma$ is generally a signed distribution given on a bounded domain $\Om\subset \RR^n$, $n\geq 2$.

In \eqref{Riccati11} the nonlinearity   $\mathcal{A}: \RR^n\times\RR^n \rightarrow \RR^n$ is a Carath\'eodory vector valued function, i.e., $\mathcal{A}(x,\xi)$ is measurable in $x$ for every $\xi$ and continuous in
$\xi$ for a.e. $x$. Moreover,  for a.e. $x$, $\mathcal{A}(x,\xi)$ is differentiable in $\xi$ away from the origin. 
Our standing assumption is that  $\mathcal{A}$ satisfies the following growth and monotonicity conditions: for some $1<p<\infty$ and $\Lambda\geq 1$ there hold
\begin{equation}\label{sublinear}
|\mathcal{A}(x,\xi)| \leq  \Lambda \m{\xi}^{p-1},  \quad |\nabla_{\xi}\mathcal{A}(x,\xi)| \leq  \Lambda\m{\xi}^{p-2}
\end{equation}
and
\begin{equation}\label{monotone}
\langle\mathcal{A}(x,\xi)-\mathcal{A}(x,\eta),\xi-\eta \rangle\geq   \Lambda^{-1} (|\xi|^2+|\eta|^2)^{\frac{p-2}{2}}|\xi-\eta|^2
\end{equation}
for any  $(\xi, \eta)\in \RR^n \times \RR^n\setminus (0,0)$
and a.e. $x \in \RR^n$.  The special case $\mathcal{A}(x, \xi)=|\xi|^{p-2}\xi$ gives 
rise to the standard $p$-Laplacian $\Delta_p u= {\rm div}\,( |\nabla u|^{p-2} \nabla u)$. Note that these conditions imply that  $\mathcal{A}(x,0)=0$ for a.e. $x\in \RR^n$, and 
\begin{equation*}
  \langle \nabla_{\xi} \mathcal{A}(x,\xi)\lambda, \lambda\rangle \geq 2^{\frac{p-2}{2}}   \Lambda^{-1}  |\xi|^{p-2} |\lambda|^{2}
\end{equation*}
for every $(\lambda,\xi)\in\RR^n \times\RR^n\setminus\{(0,0)\}$ and  a.e. $x \in\RR^n$.

More regularity conditions will be imposed later on the nonlinearity $\mathcal{A}(x,\xi)$  in the $x$-variable  and on  the boundary $\partial\Om$ of $\Om$.

One can view  \eqref{Riccati11}  as a quasilinear stationary  viscous Hamilton-Jacobi equation or  Kardar-Parisi-Zhang equation,  
which appears  in the physical theory of surface growth  \cite{KPZ, KS}.

\noindent {\bf Necessary conditions:} For $q>p-1$, it is  known   (see \cite{HMV, Ph1}) that in order for \eqref{Riccati11} to have a $u$ with  $|\nabla u|\in L^{q}_{\rm loc}(\Om)$ it is necessary that $\sigma$ be regular and small enough. In particular, if $\sigma$ is a signed measure  these necessary conditions can be quantified as
\begin{equation}\label{necmeas}
\int_{\Om} |\varphi|^{\frac{q}{q-p+1}} d\sigma \leq \Lambda^{\frac{q}{q-p+1}} \left(\frac{q-p+1}{p-1}\right)^\frac{1-p}{q-p+1} \int_{\Om} |\nabla \varphi|^{\frac{q}{q-p+1}} dx
\end{equation}
for all $\varphi\in C_0^{\infty}(\Om)$. This can be seen by  
using $|\varphi|^{\frac{q}{q-p+1}}$ as a test function in \eqref{Riccati11} and applying the first inequality in \eqref{sublinear} to get
$$ \int_{\Om} |\varphi|^{\frac{q}{q-p+1}} d\sigma \leq   \frac{\Lambda q}{q-p+1} \int_{\Om} |\nabla u|^{p-1}  |\varphi|^{\frac{p-1}{q-p+1}} |\nabla\varphi| dx-\int_{\Om} |\nabla u|^q |\varphi|^{\frac{q}{q-p+1}} dx. $$
Then by an appropriate  Young's inequality one arrives at  \eqref{necmeas} (see also \cite{Ph1} and \cite{JMV2}). Note that \eqref{necmeas} also holds
 when $\sigma$ is a distribution in $W^{-1, \, \frac{q}{p-1}}_{\rm loc}(\Om)$ in which case the left-hand side should be understood as $\langle \sigma, |\varphi|^{\frac{q}{q-p+1}}\rangle$.

Thus if $\sigma$ is a {\it nonnegative measure} (or equivalently a nonnegative distribution) compactly supported in $\Om$ then  condition \eqref{necmeas}
implies the capacitary condition
\begin{equation}\label{cappos}
\sigma(K)\leq C \, {\rm Cap}_{1, \, \frac{q}{q-p+1}}(K)
\end{equation}
for every compact set $K\subset\Om$ and  a constant $C$ independent of $K$.
Here ${\rm Cap}_{1, \, s}$, $s>1$, is the capacity associated to the Sobolev space $W^{1,\, s}(\RR^n)$ defined for each compact set $K\subset\RR^n$ by
\begin{equation*}
{\rm Cap}_{1, \, s}(K)=\inf\Big\{\int_{\RR^n}(|\nabla \varphi|^s +\varphi^s) dx: \varphi\in C^\infty_0(\RR^n),
\varphi\geq \chi_K \Big\},
\end{equation*}
where $\chi_{K}$ is the characteristic function of $K$.

Moreover, in the case of nonnegative measure datum $\sigma$, all solutions of \eqref{Riccati11} must obey the regularity condition
\begin{equation}\label{nab}
\int_{K}|\nabla u|^q dx \leq C \, {\rm Cap}_{1, \, \frac{q}{q-p+1}}(K)
\end{equation}
for every compact set $K\subset\Om$. However, unlike \eqref{cappos}, the constant  $C$ in \eqref{nab}  might depend on the distance from $K$ to the boundary of $\Omega$  (see \cite{HMV, Ph1}).

Motivated from \eqref{cappos}, we now introduce 
the following definition.

\begin{definition} Given  $s >1$ and a  domain $\Om\subset\RR^n$ we define the space $M^{1, \, s}(\Om)$ to be the set of all  signed measures
	$\mu$ with bounded total variation in $\Om$ such that the quantity $[\mu]_{M^{1,\, s}(\Om)}<+\infty$, where
	$$[\mu]_{M^{1,\, s}(\Om)}:=\sup\left\{ |\mu|(K)/{\rm Cap}_{1,\, s}(K): {\rm Cap}_{1, \, s}(K)>0 \right\},$$
	with the supremum being taken over all compact sets $K\subset \Om$.
\end{definition}

It is well-known that a measure $\mu\in M^{1,\, s}(\Om)$ if and only if the trace inequality
\begin{equation}\label{trace}
\int_{\RR^n} |\varphi|^{s} d|\mu| \leq C \int_{\RR^n} ( |\nabla \varphi|^{s} +|\varphi|^s )dx
\end{equation}
holds for all $\varphi\in C^{\infty}_0(\RR^n)$, with a constant $C$ independent of $\varphi$. Here $\mu$ is extended by zero outside $\Om$. For
this characterization see, e.g., \cite{AH}. Other characterizations are also available (see \cite{MV}).

In practice, it is useful to realize that the condition $\mu \in M^{1,\, s}(\Om)$ is satisfied if $\mu$ is a {\it function} verifying 
the Fefferman-Phong condition $\mu\in \mathcal{L}^{1+\epsilon;\, s(1+\epsilon)}(\Om)$ for some 
$\epsilon>0$ (see \cite{Fef}). Here $\mathcal{L}^{1+\epsilon;\, s(1+\epsilon)}(\Om)$ is a Morrey space (see, e.g., \cite{M-P3}). In particular, it is satisfied 
provided $\mu$ is a function in the weak Lebesgue space $L^{\frac{n}{s},\, \infty}(\Om)$, $s<n$. Another sufficient  condition is given by   $(G_1*|\mu|)^{\frac{s}{s-1}} \in \mathcal{L}^{1+\epsilon;\, s(1+\epsilon)}(\Om)$ for some $\epsilon>0$ (see \cite{MV}), where  $G_1$ is the Bessel kernel of order $1$ defined via its Fourier transform by $\widehat{G_1}(\xi)=(1+|\xi|^2)^{\frac{-1}{2}}$.

Now in view of \eqref{nab}, it is natural to look for a solution $u$ of \eqref{Riccati11} such that $|\nabla u|^q$ belongs to $M^{\frac{q}{q-p+1}}(\Om)$. In this paper, we will be interested in only  such a  space of solutions.

\noindent {\bf Sufficient conditions in capacitary terms:} There are many papers that obtain existence results for equation 
\eqref{Riccati11} under certain integrability conditions on the datum $\sigma$ which are generally not sharp. 
 The pioneering work \cite{HMV}  originally used capacities to treat   \eqref{Riccati11} in the `linear' case $p=2$    in $\RR^n$ ($q>1$), or   in a bounded domain $\Om$ ($q>2$).   For $p> 2-\frac{1}{n}$ and $q\geq 1$,  it was shown in  \cite{Ph5, 55Ph2-2} (see also \cite{GMP, PhJM} for the sub-critical case $p-1<q < n(p-1)/(n-1)$) that, under certain regularity conditions on $\mathcal{A}$ and  $\partial\Om$, if $\sigma$ is a finite  {\it signed measure}   in $M^{\frac{q}{q-p+1}}(\Om)$, with  $[\sigma]_{M^{\frac{q}{q-p+1}}(\Om)}$ being  sufficiently small, 
 then equation \eqref{Riccati11} admits a solution
$u\in W^{1, \, q}_0(\Om)$ such that $|\nabla u|^q \in M^{\frac{q}{q-p+1}}(\Om)$.
 Similar existence results have recently been extended to the case $\frac{3n-2}{2n-1}<p\leq 2-\frac{1}{n}$, $q\geq 1$, in  \cite{QH4} and to the case $1<p\leq \frac{3n-2}{2n-1}$, $q\geq 1$, in \cite{QH5}. We also mention that the earlier work \cite{Ph1, Ph1-1}  covers all $p>1$ but only for $q>p$.

We observe that whereas the existence results of \cite{HMV,Ph5,QH4,QH5, Ph1} are sharp when $\sigma$ is a nonnegative measure,   
they   could not be applied to a large class  distributional data $\sigma$ with strong oscillation. Take for example 
the function 
$$f(x)=|x|^{-\epsilon-s}\sin(|x|^{-\epsilon}),$$
where  $s=q/(q-p+1)$ and $\epsilon>0$ such that $\epsilon+s<n$. Then $\sigma= |f(x)|dx$ fails to satisfy the capacitary inequality \eqref{cappos}, but it is possible to show that the equation $$-\Delta_p u=|\nabla u|^q +\lambda f, \qquad q \geq p,$$ admits a solution $u\in W^{1,q}_0(B_1(0))$ provided $|\lambda|$ is sufficiently small.  
For this see \cite{M-P3} which addresses oscillatory data in the Morrey space framework. See also \cite{AP1,AP2,FM1, FM2} in   which  the  case $q=p$ is considered.  Note  that in this special case, the Riccati type equation
$-{\rm div}\,\mathcal{A}(x, \nabla u)= |\nabla u|^p + \sigma$
is strongly related to the Schr\"odinger type equation $-{\rm div}\,\mathcal{A}(x, \nabla u)=\sigma |u|^{p-2}u$ (see \cite{BH}).
This relation has been employed in an essential way in \cite{JMV1, JMV2} to study the existence of  local solutions in this case.
Here by a local solution we mean one that belongs to $W^{1, \, p}_{\rm loc}(\Om)$ and has no pre-specified boundary condition.

\noindent {\bf Main results:} The first main result of this paper is to treat \eqref{Riccati11} with oscillatory data in the framework of the natural space $M^{1,\frac{q}{q-p+1}}(\Om)$. This provides non-trivial improvements of the results of \cite{HMV,Ph5,QH4,QH5, Ph1} and \cite{M-P3} at least in the case $q>p$. 
We first observe the following  necessary condition on $\sigma$ so that \eqref{Riccati11} has a solution $u$ such that $|\nabla u|^q\in M^{1,\frac{q}{q-p+1}}(\Om)$.

\begin{theorem}\label{necdis}
Let $p>1$,  $q\geq 1$,  and let  $\mathcal{A}$ satisfy the first inequality in \eqref{sublinear}.
Suppose that  $\sigma$ is a distribution in a bounded domain $\Om$ such that the Riccati type equation
\begin{equation}\label{sigmaE}
 -{\rm div}\,\mathcal{A}(x, \nabla u)= |\nabla u|^q + \sigma  \quad \text{in} ~\mathcal{D}'(\Omega)
\end{equation}
admits a  solution $u\in W^{1, \, q}(\Om)$ with $|\nabla u|^q \in M^{1,\, \frac{q}{q-p+1}}(\Om)$. Then  there exists a vector field ${\bf f}$ on $\Om$ such
that $\sigma={\rm div}\, {\bf f}$ and $|{\bf f}|^{\frac{q}{p-1}}\in M^{1,\, \frac{q}{q-p+1}}(\Om)$.  In particular, we have  $\sigma\in W^{-1, \, \frac{q}{p-1}}(\Om)$, and moreover
\begin{equation}\label{sigmatrace}
\left |\langle \sigma, |\varphi|^{\frac{q}{q-p+1}} \rangle\right| \leq C \int_{\Om} |\nabla \varphi|^{\frac{q}{q-p+1}} dx
\end{equation}
for all $\varphi\in C_0^{\infty}(\Om)$, with a constant $C$ independent of $\varphi$.
\end{theorem}

Conversely, when $q>p$ we obtain the following existence result.

\begin{theorem}\label{main-Ric}
Let $1<p<q<\infty$, $R_0>0$, and assume that  $\mathcal{A}$ satisfies \eqref{sublinear}-\eqref{monotone}.
Then there exists a constant  $\delta=\delta(n, p,  \Lambda, q)\in (0, 1)$ such that the following holds.
Let
$\om\in M^{1,\, \frac{q}{q-p+1}}(\Om)$ and let ${\bf f}$ be a vector field on $\Om$ such that $|{\bf f}|^{\frac{q}{p-1}} \in M^{1,\, \frac{q}{q-p+1}}(\Om)$.
Assume that   $\Om$ is $(\delta, R_0)$-Reifenberg flat and that $\mathcal{A}$ satisfies the $(\delta, R_0)$-BMO condition. Then there exists a positive  constant
$c_0=c_0(n, p, \Lambda, q, {\rm diam}(\Om), {\rm diam}(\Om)/R_0)$ such that whenever
\begin{equation*}
[\om]^{\frac{q}{p-1}}_{M^{1, \frac{q}{q-p+1}}(\Om)} + [|{\bf f}|^{\frac{q}{p-1}}]_{M^{1, \frac{q}{q-p+1}}(\Om)}\leq c_0,
\end{equation*}
there exists a  solution $u\in W_0^{1, \, q}(\Om)$ to the Riccati type equation
\begin{eqnarray}\label{Riccati}
\left\{ \begin{array}{rcl}
 -{\rm div}\,\mathcal{A}(x, \nabla u)&=& |\nabla u|^q + \om+ {\rm div}\, {\bf f}  \quad \text{in} ~\Omega, \\
u&=&0  \quad \text{on}~ \partial \Omega,
\end{array}\right.
\end{eqnarray}
with $|\nabla u|^q \in M^{1,\, \frac{q}{q-p+1}}(\Om)$.
\end{theorem}
 
 \begin{remark}
 	Under a slightly different condition on $\mathcal{A}(x,\xi)$, it is possible to use the results of \cite{AP00, AP01} and the method of this paper to extend  Theorem \ref{main-Ric} to the end-point case $q=p$. However, this case has been treated  in \cite{AP1} by  using a different method (see also \cite{AP2}). 
 \end{remark}
 
 The notion of   $(\delta, R_0)$-Reifenberg flat domains mentioned in Theorem \ref{main-Ric} is made precise by the following definition.  
 \begin{definition}\label{defiRei}
 	Given $\delta\in (0,1)$ and $R_0>0$, we say that $\Omega$ is a $(\delta, R_0)$-Reifenberg flat domain if for every $x_0\in \partial \Omega$
 	and every $r\in (0, R_0]$, there exists a
 	system of coordinates $\{ y_{1}, y_{2}, \dots,y_{n}\}$,
 	which may depend on $r$ and $x_0$, so that  in this coordinate system $x_0=0$ and that
 	\[
 	B_{r}(0)\cap \{y_{n}> \delta r \} \subset B_{r}(0)\cap \Omega \subset B_{r}(0)\cap \{y_{n} > -\delta r\}.
 	\]
 \end{definition}

Examples of such domains   include those with $C^1$ boundaries or Lipschitz domains with
sufficiently small Lipschitz constants. They also  include certain domains  with fractal boundaries.

On the other hand, the $(\delta, R_0)$-BMO condition  imposed on $\mathcal{A}(x,\xi)$  allows  it to have  small jump discontinuities in the $x$-variable. 
More precisely, given two  positive numbers $\delta$ and $R_0$, we say that $\mathcal{A}({x, \xi})$ satisfies the $(\delta, R_0)$-BMO condition if
\begin{equation*}
[\mathcal{A}]_{R_0}:=\sup_{y\in\RR^n, \, 0<r\leq R_0 }  \fint_{B_{r}(y)}\Upsilon(\mathcal{A},
B_{r}(y))(x) dx  \leq \delta,
\end{equation*}
where  for each ball $B=B_r(y)$ we let
\[
\Upsilon(\mathcal{A},B)(x) := \sup_{\xi \in \mathbb{R}^{n}\setminus \{0\}} \frac{|\mathcal{A}({x, \xi}) - \overline{\mathcal{A}}_{B}({\xi})|}{|\xi|^{p-1}},
\]
with $\overline{\mathcal{A}}_{B}(\xi) = \fint_{B}\mathcal{A}(x, \xi)dx$.
Thus one can think of the $(\delta, R_0)$-BMO condition as  an appropriate substitute for
the Sarason VMO  condition.

 The second  main result of the paper is to treat \eqref{Riccati11} for the case  $p>\frac{3n-2}{2n-1}$, $p-1<q<1$,  and  $\sigma$ is a signed measure compactly supported in $\Om$.
 This extends the  results of \cite{QH4} to the sublinear range $p-1<q<1$, which cannot be dealt with by the method of \cite{QH4} due to  the lack of convexity. However, here we  assume that $\mathcal{A}(x,\xi)$ is H\"older
 continuous in the $x$-variable, i.e.,
 \begin{equation}\label{Holder}
 |\mathcal{A}(x,\xi)-\mathcal{A}(x_0,\xi)|\leq \Lambda |x-x_0|^{\theta} |\xi|^{p-1}
 \end{equation}
 for some $\theta\in (0,1)$ and all $x, x_0, \xi\in\RR^n$. We note that   this regularity assumption can be relaxed  by using a weaker Dini's condition as in \cite{HP}. 
 Moreover, for $\Om$ we further assume the following integrability condition
  (besides the $(\delta, R_0)$-Reifenberg flatness condition):
  \begin{equation}\label{integcond}
  \int_\Om d(x)^{-\epsilon_0} dx <+\infty
  \end{equation}
for some $\epsilon_0>0$. Here $d(x)$ is the distance from $x$ to $\partial\Om$, i.e.,
$d(x)=\inf\{|x-y|: y\in\partial\Om\}$. It is not clear to us if the  $(\delta, R_0)$-Reifenberg flatness condition for a sufficiently small $\delta$ will imply 
\eqref{integcond}. Note that $\eqref{integcond}$  holds (even with any $0<\epsilon_0<1$) for any bounded Lipschitz domain. More generally, $\eqref{integcond}$
 holds for some  $\epsilon_0>0$ provided we can find an $\epsilon>0$ such that  
$$ \left| \{x\in \Omega: \tau <d(x) \leq 2\tau\} \right| \leq C \tau^{\epsilon}$$
holds for all small $\tau>0$.

\begin{theorem}\label{second-main-Ric}
	Let $p>\frac{3n-2}{2n-1}$, $p-1<q<1$, $R_0>0$, and assume that  $\mathcal{A}$ satisfies \eqref{sublinear}, \eqref{monotone}, and \eqref{Holder}.
	Suppose that \eqref{integcond} holds for an $\epsilon_0>0$ and that	$\om\in M^{1,\, \frac{q}{q-p+1}}(\Om)$ with ${\rm supp}(\om)\Subset\Om$.
	Then there exists a constant  $\delta=\delta(n, p,  \Lambda, q, \epsilon_0)\in (0, 1)$ such that the following holds. If $\Om$ is $(\delta, R_0)$-Reifenberg flat, then there exists a positive  constant
	$$c_0=c_0(n, p, \Lambda, q, \theta, \epsilon_0, {\rm diam}(\Om), {\rm diam}(\Om)/R_0, {\rm dist}({\rm supp}(\om),\partial\Om))$$ such that whenever
	\begin{equation}\label{c0forom}
	[\om]^{\frac{q}{p-1}}_{M^{1, \frac{q}{q-p+1}}(\Om)}\leq c_0,
	\end{equation}
	there exists a  renormalized  solution $u$, with $|\nabla u|^q \in M^{1,\, \frac{q}{q-p+1}}(\Om)$, to the Riccati type equation
	\begin{eqnarray}\label{Riccati2}
	\left\{ \begin{array}{rcl}
	-{\rm div}\,\mathcal{A}(x, \nabla u)&=& |\nabla u|^q + \om  \quad \text{in} ~\Omega, \\
	u&=&0  \quad \text{on}~ \partial \Omega.
	\end{array}\right.
	\end{eqnarray}
\end{theorem}

We refer to \cite{DMOP} for the notion of renormalized solutions. Note that in the case $p\leq 2-\frac{1}{n}$ the gradients of such solutions should be interpreted  appropriately.

\begin{remark} It is worth mentioning that  the case $p>2-\frac{1}{n}$ and $p-1<q<1$, which is a sub-critical case, has been addressed in \cite{GMP, PhJM} by  different methods that require no compact support condition on $\om$. However, our proof of Theorem \ref{second-main-Ric} produces a solution to \eqref{Riccati2} whose gradient is well controlled pointwise. Moreover, our proof also works in the super-linear case $q\geq 1$ that was considered  earlier in \cite{QH4}.
\end{remark}	

\section{Proof of Theorems \ref{necdis} and \ref{main-Ric}}

In this section we prove Theorems \ref{necdis} and \ref{main-Ric}. We begin with the proof of Theorem \ref{necdis}.
\begin{proof}[{\bf Proof of Theorem \ref{necdis}}.]  Here we employ an idea of \cite{JMV1,JMV2} that treated the case $q=p$. 
	Let $B$ is a ball of radius ${\rm diam}(\Om)$ containing $\Om$ and let $G(x,y)$ be the Green function with zero boundary condition associated to $-\Delta$ on $B$.
	Then it follows that  $$|\nabla u(x)|^q=-{\rm div}\,  \int_{B}\nabla_x G(x,y) |\nabla u(y)|^q\chi_{\Om}(y) dy  \quad {\rm in~} \mathcal{D}'(\Om).$$

	Thus by \eqref{sigmaE} we have that   $\sigma = {\rm div}\, {\bf f}$ in $\mathcal{D}'(\Om)$ with
	$${\bf f}=-\mathcal{A}(x, \nabla u) +\int_{B}\nabla_x G(x,y) |\nabla u(y)|^q\chi_{\Om}(y) dy.$$
	
	Note that by the first inequality in \eqref{sublinear} we find
	$$\left[|\mathcal{A}(x, \nabla u)|^{\frac{q}{p-1}}\right]^{\frac{p-1}{q}}_{M^{1,\, \frac{q}{q-p+1}}(\Om)} \leq \Lambda [|\nabla u|^q]^{\frac{p-1}{q}}_{M^{1,\, \frac{q}{q-p+1}}(\Om)}.$$

	On the other hand, using the pointwise estimate
	\begin{equation}\label{Greenest}
	|\nabla_x G(x, y)|\leq C(n, {\rm diam}(\Om)) |x-y|^{1-n} \qquad \forall x, y\in B, x\not=y,
	\end{equation}
	and \cite[Corollary 2.5]{Ph1} we obtain
	$$\left[\left|\int_{B}\nabla_x G(\cdot ,y) |\nabla u(y)|^q\chi_{\Om}(y) dy\right|^{\frac{q}{p-1}}\right ]^{\frac{p-1}{q}}_{M^{1, \frac{q}{q-p+1}}(\Om)}\leq C\, [|\nabla u|^q]_{M^{1, \, \frac{q}{q-p+1}}(\Om)}.$$
	
	These show that $|{\bf f}|^{\frac{q}{p-1}}\in M^{1, \, \frac{q}{q-p+1}}(\Om)$ with the estimate
	$$ \left[|{\bf f}|^{\frac{q}{p-1}}\right]^{\frac{p-1}{q}}_{M^{1, \frac{q}{q-p+1}}(\Om)}\leq C\left(   \left[|\nabla u|^q\right]^{\frac{p-1}{q}}_{M^{1,\, \frac{q}{q-p+1}}(\Om)} + \left[|\nabla u|^q \right]_{M^{1, \, \frac{q}{q-p+1}}(\Om)}\right).$$

	Finally, given any $\varphi\in C_0^{\infty}(\Om)$ we have
	\begin{eqnarray*}
		\left|\langle \sigma, |\varphi|^{\frac{q}{q-p+1}} \rangle\right| &=& \left| \int_{\Om}{\bf f}\cdot \nabla (|\varphi|^{\frac{q}{q-p+1}}) dx\right| \leq \frac{q}{q-p+1} \int_{\Om}  |{\bf f}| |\varphi|^{\frac{p-1}{q-p+1}} |\nabla \varphi|dx\\
		&\leq& \frac{q}{q-p+1} \left(\int_{\Om}  |{\bf f}|^{\frac{q}{p-1}} |\varphi|^{\frac{q}{q-p+1}} dx\right)^{\frac{p-1}{q}} \left(\int_{\Om} |\nabla \varphi|^{\frac{q}{q-p+1}}dx\right)^{\frac{q-p+1}{q}}\\
		&\leq& C \int_{\Om} |\nabla \varphi|^{\frac{q}{q-p+1}}dx.
	\end{eqnarray*}
	Here the last inequality follows since by  \eqref{trace} and Poincar\'e's inequality we have
	$$\int_{\Om}  |{\bf f}|^{\frac{q}{p-1}} |\varphi|^{\frac{q}{q-p+1}} dx \leq C({\rm diam}(\Om)) \int_{\Om} |\nabla \varphi|^{\frac{q}{q-p+1}}dx. $$
	
	Thus  \eqref{sigmatrace} is verified, which completes the proof of the theorem.
\end{proof}

In order to Theorem \ref{main-Ric}, we need the following equi-integrability result.

\begin{lemma}\label{compactness}	
	 For each $j=1,2,3, \dots$, let ${\bf f}_j \in L^{\frac{q}{p-1}}(\Om, \RR^n)$, $q>p$, and $u_j\in W^{1,q}_0(\Om)$ be the solution of  $${\rm div} \mathcal{A}(x, \nabla u)={\rm div}\, {\bf f}_{j} \quad {\rm in~} \Om.$$ 
	Assume that 
	$\{|{\bf f}_j|^{\frac{q}{p-1}}\}_{j}$ is a bounded and equi-integrable subset of  $L^1(\Omega)$.
	Then, there exists  $\delta=\delta(n,p,\Lambda, q)\in (0,1)$ such that if $\Omega$ is  $(\delta,R_0)$-Reifenberg flat  and $[\mathcal{A}]_{R_0}\le \delta$ for some $R_0>0$, then 
	the set $\{|\nabla u_j|^q\}_j$ is also a bounded and equi-integrable subset of  $L^1(\Omega)$.
\end{lemma}

\begin{proof}
	By de la Vall\'ee-Poussin Lemma on equi-integrability  we can find an increasing  and convex  function 
	$G:[0,\infty)\rightarrow [0,\infty)$ with $G(0)=0$ and $\lim_{t\rightarrow \infty} \frac{G(t)}{t}=\infty,$
	such that 
	\begin{equation*}
	\sup_{j} \int_{\Om} G(|{\bf f}_j|^{\frac{q}{p-1}}) dx \leq C.
	\end{equation*}
	Moreover, we may assume that $G$ satisfies a $\Delta_2$ (moderate growth) condition (see, e.g., \cite{Mey}): there exists $c_1>1$ such that 
	$$G(2t)\leq c_1\, G(t)\qquad \forall t\geq 0.$$ It follows that the function  $\Phi(t):=G(t^{q/p})$  also satisfies a $\Delta_2$  condition since
	$$\Phi(2t)=G(2^{q/p} t^{q/p})\leq G(2^{[q/p]+1} t^{q/p})\leq (c_1)^{[q/p]+1}\Phi(t),$$
	where $[q/p]$ is the integral part of $q/p$. 
	
	On the other hand, as $G$ is convex and $G(0)=0$, for $c_2=2^{\frac{p}{q-p}}>1$ we have
	$$\Phi(t)=G( c_2^{-q/p}(c_2t)^{q/p})\leq c_2^{-q/p} G((c_2t)^{q/p})=\frac{1}{2c_2} \Phi(c_2t).$$
	In other words, $\Phi$ satisfies a $\nabla_2$ condition.
	
	Also, by the above properties of $G$ we have that $\Phi$ is an 
	increasing  and convex Young function, i.e., 
	$$\Phi(0)=0, \quad \lim_{t\rightarrow 0^{+}} \frac{\Phi(t)}{t}=0,\quad {\rm and} \quad \lim_{t\rightarrow \infty} \frac{\Phi(t)}{t}=\infty.$$
	
	With these properties of $\Phi$, by  the main result of \cite{BR2} (see also \cite{BW4}),
	we have that 
	\begin{equation*}
	\sup_{j} \int_{\Om} \Phi(|\nabla u_j|^{p}) dx = \sup_{j} \int_{\Om} G(|\nabla u_j|^{q}) dx\leq C.
	\end{equation*}
	Here the constant $C$ depends   only on $n, p, q, G, \Lambda, \Om$, and $\delta$.
	Hence by de la Vall\'ee-Poussin Lemma, it follows that the sequence $\{|\nabla u_j|^{q}\}_{j}$ is equi-integrable in $\Om$.
\end{proof}	

We now recall that $G_1$ is the Bessel kernel of order $1$. For any nonnegative measure $\nu$, we define a  Bessel potential of $\nu$ by 
$${\bf G}_{1}(\nu)(x):=G_1*\nu(x)=\int_{\RR^n} G_1(x-y) d\nu(y), \quad x\in\RR^n.$$

\begin{lemma}\label{interm}
	Let $q>p>1$ and suppose that
	$\mu\in M^{1,\, \frac{q}{q-p+1}}(\Om)$ and that ${\bf g}$ is a vector field on $\Om$ such that $|{\bf g}|^{\frac{q}{p-1}} \in M^{1,\, \frac{q}{q-p+1}}(\Om)$.
	There exists a constant  $\delta=\delta(n, p,  \Lambda, q)\in (0, 1)$ such that if  $\Om$ is $(\delta, R_0)$-Reifenberg flat and 
	$[\mathcal{A}]_{R_0}\leq \delta$ for some $R_0>0$ then 	the equation
	\begin{eqnarray}\label{U-equa}
		\left\{ \begin{array}{rcl}
			{\rm div}\,\mathcal{A}(x, \nabla U)&=& \mu + {\rm div}\, {\bf g}  \quad \text{in} ~\Omega, \\
			u&=&0  \quad \text{on}~ \partial \Omega
		\end{array}\right.
	\end{eqnarray}
	admits a unique solution $U\in W^{1,\, q}_0(\Om)$ with
	\begin{equation}\label{pointwiseB}
	{\bf G}_{1}(|\nabla U|^q)\leq C\, [{\bf G}_{1}(|{\bf g}|^{\frac{q}{p-1}}) + [\mu]^{\frac{q-p+1}{p-1}}_{M^{1,\frac{q}{q-p+1}}(\Om)}\  {\bf G}_{1}(|\mu|) ] \quad {\rm a.e.~in~} \RR^n.
	\end{equation}
	Here  $U$, ${\bf g}$, and $\mu$ are extended  by zero outside $\Om$. The constant $C$ in \eqref{pointwiseB} depends only on $n, p, \Lambda, q, {\rm diam}(\Om)$,  and ${\rm diam}(\Om)/R_0$.
\end{lemma}
\begin{proof}
	Again, let $B$ is a ball of radius ${\rm diam}(\Om)$ containing $\Om$ and let $G(x,y)$ be the Green function with zero boundary condition associated to $-\Delta$ on $B$.
	Then we can write $\mu=-{\rm div}\, {\bf h_{\mu}}$ in $\mathcal{D}'(\Om),$
	where ${\bf h_{\mu}}$ is a gradient vector field on $B$ given by
	\begin{equation}\label{Gree}
	{\bf h_{\mu}}(x)=\int_{B}\nabla_x G(x,y)d\mu(y). 
	\end{equation}
	
	In what follows, we say that a function $w\in {\bf A}_1$ if $w\in L^1_{\rm loc}(\RR^n)$, $w\geq 0$, and  
	$$\sup_{r>0}\fint_{B_r(x)} w(y)dy \leq A w(x) \quad {\rm for~ a.e.~} x\in\RR^n. $$
	The least possible constant $A$ in the above inequality is called the ${\bf A}_1$ constant of $w$ and is denoted by $[w]_{\mathbf{A}_1}$.\\

	Note then that  by \cite[Theorem 1.10]{M-P3}, 	for any weights $w\in {\bf A}_1$,  there exists a constant  $\delta=\delta(n, p,  \Lambda, q,[w]_{\mathbf{A}_1})\in (0, 1)$ such that if  $\Om$ is $(\delta, R_0)$-Reifenberg flat and 
	$[\mathcal{A}]_{R_0}\leq \delta$ then 	 equation \eqref{U-equa}
		admits a unique solution $U\in W^{1,\, q}_0(\Om)$ such that 
	\begin{equation}\label{Uw}
	\int_{\Om}|\nabla U|^q w dx \leq C \int_{\Om} |{\bf g}-{\bf h_\mu}|^{\frac{q}{p-1}} w dx.
	\end{equation}
Moreover, the constant $C$ in \eqref{Uw} depends on $w$ only through $[w]_{\mathbf{A}_1}$. 
	
	We now observe from the asymptotic behavior of $G_1$ (see \cite[Section 1.2.4]{AH}) that the function  
	$w(x)={\bf G}_1(g)(x)$,	where $g$ is any nonnegative and bounded function with compact support, satisfies the following \emph{local} ${\bf A}_1$ condition
	$$\sup_{0<r\leq 1}\fint_{B_r(x)} w(y)dy \leq A w(x) \quad {\rm for~ a.e.~} x\in\RR^n. $$
	The constant $A$ is independent of $g$. Thus by \cite[Lemma 1.1]{Ryc}  there exists a weight $\overline{w}\in A_1$ such that $w=\overline{w}$ in $B$ and 
	$[\overline{w}]_{{\bf A}_1}\leq C=C(n,{\rm diam}(\Om),A)$.
	Then using $\overline{w}$ in \eqref{Uw} and applying  Fubini's Theorem   we find
	
	$$\int_{\RR^n} {\bf G}_{1}(|\nabla U|^q\chi_{\Om}) g dx \leq C \int_{\RR^n} {\bf G}_{1}(|{\bf g}-{\bf h}_\mu|^{\frac{q}{p-1}}\chi_{\Om}) g dx.$$
	
	Due to the arbitrariness of $g$, this yields  
	\begin{equation}\label{I1nabU}
	{\bf G}_{1}(|\nabla U|^q\chi_{\Om})\leq C\, {\bf G}_{1}(|{\bf g}-{\bf h}_\mu|^{\frac{q}{p-1}}\chi_{\Om}) \quad {\rm a.e.~in~} \RR^n
	\end{equation}
	for a constant $C$ that depends only on $n, p, \Lambda, q, {\rm diam}(\Om)$,  and ${\rm diam}(\Om)/R_0$.
	
	Note that by \eqref{Gree} and  the pointwise estimate \eqref{Greenest} it follows that
	\begin{equation}\label{hcontrol}
	|{\bf h}_\mu(x)| \leq C\, {\bf G}_{1}(|\mu|)(x) \quad {\rm a.e.~in~} \RR^n.
	\end{equation}

	On the other hand, by \cite[Theorem 1.2]{MV}	we find
	\begin{equation}\label{lap}
	 {\bf G}_{1}[{\bf G}_{1}(|\mu|)^{\frac{q}{p-1}}](x)\leq C [\mu]^{\frac{q-p+1}{p-1}}_{M^{1,\frac{q}{q-p+1}}(\Om)}\  {\bf G}_{1}(|\mu|)(x) \quad {\rm a.e.~in~} \RR^n.
	\end{equation}
	
Thus in  view of \eqref{hcontrol} we see that 
	
	\begin{equation}\label{Vercon}
	{\bf G}_{1}[|{\bf h}_\mu|^{\frac{q}{p-1}}] \leq C [\mu]^{\frac{q-p+1}{p-1}}_{M^{1,\frac{q}{q-p+1}}(\Om)}\  {\bf G}_{1}(|\mu|)(x) \quad {\rm a.e.~in~} \RR^n.
	\end{equation}

	Combining \eqref{I1nabU} and \eqref{Vercon} we arrive at the  pointwise estimate \eqref{pointwiseB} as desired.
\end{proof}

We are now ready  to prove Theorem \ref{main-Ric}. 

\begin{proof}[{\bf Proof of Theorem \ref{main-Ric}}] Let $\om$ and ${\bf f}$ be as in the theorem.
	Our strategy is to apply Schauder Fixed Point Theorem to the following closed and convex subset of  $W_0^{1,\, q}(\Om)$:
	$$E:= \left\{v\in W_0^{1,\, q}(\Om): {\bf G}_1(|\nabla v|^q) \leq T\, {\bf G}_1[|{\bf f}|^\frac{q}{p-1} + {\bf G}_1(|\om|)^{\frac{q}{p-1}}] {\rm ~a.e.} \right\},$$	
	where $T>0$ is to be chosen.	
	
	Note that by \eqref{lap} we have 
	$${\bf G}_1 [{\bf G}_1(|\om|)^{\frac{q}{p-1}}] \leq C [\om]^{\frac{q-p+1}{p-1}}_{M^{1,\frac{q}{q-p+1}}(\Om)}\ {\bf G}_1(|\om|).$$

	Thus by Theorems 1.1 and 1.2 of \cite{MV} (see also \cite[Theorem 2.3]{Ph1}), from the definition of $E$ we obtain for any $v\in E$, 
	\begin{equation*}
	[|\nabla v|^q]_{M^{1, \frac{q}{q-p+1}}(\Om)}
	\leq C_0 T  \left[[\om]^{\frac{q}{p-1}}_{M^{1,\, \frac{q}{q-p+1}}(\Om)} + [|{\bf f}|^{\frac{q}{p-1}}]_{M^{1, \frac{q}{q-p+1}}(\Om)}\right]
	\end{equation*}
	for a constant $C_0$ depends only on $n, p, \Lambda, q, {\rm diam}(\Om)$,  and ${\rm diam}(\Om)/R_0$.
	
	Therefore, if we assume that
	\begin{equation*}
	[\om]_{M^{1, \frac{q}{q-p+1}}(\Om)}^{\frac{q}{p-1}} + [|{\bf f}|^{\frac{q}{p-1}}]_{M^{1, \frac{q}{q-p+1}}(\Om)}\leq c_0,
	\end{equation*}
	where $c_0$ is to be determined, then we have for any $v\in E$, 
	\begin{equation}\label{nabva}
	[|\nabla v|^q]_{M^{1, \frac{q}{q-p+1}}(\Om)} 	\leq c_0 C_0 T. 
	\end{equation}

	Let $S: E\rightarrow W_0^{1,\, q}(\Om)$ be defined by 
	$S(v)=u$ where $u\in W_0^{1,\, q}(\Om)$ is the unique   solution of
	\begin{eqnarray*}
	\left\{\begin{array}{rcl}
	-{\rm div}\, \mathcal{A}(x, \nabla u) &=& |\nabla v|^q + \om + {\rm div}\, {\bf f}\quad {\rm in}~ \Om,\\
	u&=&0\quad {\rm on}~ \partial\Om.
	\end{array}
	\right.
	\end{eqnarray*}
	
	We  claim that there are  $T>0$ and $c_0>0$ such that  
	$S: E \rightarrow E$.

	 By Lemma \ref{interm} we may assume that 
	\begin{equation}\label{ZZZ}
	{\bf G}_{1}(|\nabla S(v) |^q)\leq C_1 \left[{\bf G}_{1}(|{\bf g}|^{\frac{q}{p-1}}) + [|\nabla v|^q]^{\frac{q-p+1}{p-1}}_{M^{1,\frac{q}{q-p+1}}(\Om)}\ {\bf G}_{1}(|\nabla v|^q) \right] \quad {\rm a.e.~in~} \RR^n,
	\end{equation}
	where ${\bf g}={\bf f}-{\bf h}_{\om}$ and ${\bf h}_{\om}$ is the gradient vector field associated to $\om$ as in the proof of Lemma  \ref{interm}. 
	
	We next  note from \eqref{hcontrol} that 
	\begin{equation}\label{gcon}
	|{\bf g}|^{\frac{q}{p-1}}\leq C_2 [|{\bf f}|^{\frac{q}{p-1}} + {\bf G}_1(|\om|)^{\frac{q}{p-1}}] \quad {\rm a.e.}
	\end{equation}

	Moreover, 	in view of  \eqref{nabva} we have 
	\begin{equation}\label{small}
	[|\nabla v|^q]^{\frac{q-p+1}{p-1}}_{M^{1,\frac{q}{q-p+1}}(\Om)}\,  {\bf G}_{1}(|\nabla v|^q)  \leq  (c_0C_0T)^{\frac{q-p+1}{p-1}} T\, {\bf G}_1[|{\bf f}|^{\frac{q}{p-1}} + {\bf G}_{1}(|\om|)^{\frac{q}{p-1}}] .
	\end{equation}

	Combining \eqref{ZZZ}, \eqref{gcon}, and \eqref{small} yields 
	\begin{equation*}
	{\bf G}_{1}(|\nabla S(v) |^q)\leq (\max\{C_1, C_2\}+1)^2  \left((c_0C_0T)^{\frac{q-p+1}{p-1}} T+1\right) {\bf G}_{1}[|{\bf f}|^{\frac{q}{p-1}} + {\bf G}_1(|\om|)^{\frac{q}{p-1}}]. 
	\end{equation*}

	We now choose $T=2(\max\{C_1, C_2\}+1)^2$ and then choose $c_0>0$ so that $ (c_0C_0T)^{\frac{q-p+1}{p-1}} T \leq 1$. Then it follows  that 
	\begin{equation*}
	{\bf G}_{1}(|\nabla S(v) |^q) \leq T \left[ {\bf G}_{1}(|{\bf f}|^{\frac{q}{p-1}}) +  {\bf G}_{1}(|\om|)  \right],
	\end{equation*}
	and thus  $S(v)\in E$ as desired.

	We next show that the set  $S(E)$ is precompact 	in the strong topology of $W_{0}^{1,\, q}(\Om)$.  Let 
	$u_k=S(v_k)$ where $\{v_k\}$ is a sequence in $E$. We have
	\begin{eqnarray*}
	\left\{\begin{array}{rcl}
	-{\rm div}\, \mathcal{A}(x, \nabla u_k) &=& |\nabla v_k|^q + \om + {\rm div }\, {\bf f} \quad {\rm in}~ \Om,\\
	u_k&=&0\quad {\rm on}~ \partial\Om.
	\end{array}
	\right.
	\end{eqnarray*}
	
	As $|\nabla v_k|^q + \om + {\rm div }\, {\bf f}={\rm div}\, ({\bf f}- {\bf h}_\om -{\bf h}_{|\nabla v_k|^q})$ in $\mathcal{D}'(\Om)$, where 
	\begin{align*}
	|{\bf h}_\om| +|{\bf h}_{|\nabla v_k|^q})|&\leq C {\bf G}_1(|\om|+ |\nabla v_k|^q)\\
	&\leq C[{\bf G}_1(|\om|) + T\, {\bf G}_1[|{\bf f}|^\frac{q}{p-1} + {\bf G}_1(|\om|)^{\frac{q}{p-1}}]] \\
	& \leq C[{\bf G}_1(|\om|) +  {\bf G}_1(|{\bf f}|^\frac{q}{p-1})],
	\end{align*}
	we may apply Lemma \ref{compactness} to see that  $\{|\nabla u_k|^q\}$ is a bounded and equi-integrable subset of $L^1(\Om)$.

	On the other hand, by \cite[Theorem 2.1]{BM} there exists a subsequence $\{u_{k'}\}$ and a function $u\in W^{1,\, q}_0(\Om)$ such that
	$$\nabla u_{k'} \rightarrow \nabla u$$ 
	a.e. in $\Om$. Thus  Vitali Convergence Theorem  yields that $u_{k'}\rightarrow u$ in $W^{1,\, q}_0(\Om)$ as desired.

	Similarly, by uniqueness we see that the map $S$ is continuous on $E$ (in the strong topology of $W_0^{1,\, q}(\Om)$). 
	Then  by  Schauder Fixed Point Theorem, S has a fixed point in $E$, which gives a solution $u$ to  problem \eqref{Riccati}. This completes the proof of the theorem.
\end{proof}

\section{Proof of Theorem \ref{second-main-Ric}}

For any nonnegative measure $\nu$ we define
\begin{equation*}
\mathbf{P}^R[\nu](x)=\left(\int_{0}^{R}\left(\frac{\nu(B_r(x))}{r^{n-1}}\right)^\beta \frac{dr}{r}\right)^{\frac{1}{\beta(p-1)}}, \quad R=2{\rm diam}(\Omega),
\end{equation*}
where  $\beta=1$ if $p>2-1/n$ and $\beta$ is any number in $\left(0,\frac{(p-1)n}{n-1}\right)$ if $\frac{3n-2}{2n-1}<p\leq 2-1/n$.  For $\kappa>0$, we also let 
$$
\mathbf{T}[\nu](x)=d(x)^{-\kappa}\mathbf{P}^R[\nu](x)\chi_{\Om}(x),
$$ 
where recall that $d(x)$ is the distance from $x$ to $\partial\Om$.

It is clear that  if $\epsilon_0$ is a positive number for which \eqref{integcond} holds then 
for any $0<\kappa \leq \frac{\epsilon_0}{4n}$,
\begin{equation}\label{Z2}
\|d^{-\kappa}\|_{L^{2n}(\Omega)}\leq C.
\end{equation}

On the other hand, note that for any $f\in L^{2n}(\Omega)$, 
$$
\|\mathbf{P}^R[|f|]\|_{L^\infty(\mathbb{R}^n)}\leq C(R,\beta)\|f\|_{L^{2n}(\Omega)}^{\frac{1}{p-1}}.
$$

Thus we have that, for any $0<\kappa \leq \frac{\epsilon_0}{4n}$,  
\begin{equation}\label{Z1}
\|\mathbf{P}^R[d(\cdot)^{-\kappa }\chi_\Om(\cdot)]\|_{L^\infty(\mathbb{R}^n)}\leq C.
\end{equation}

We now record the following result that was obtained in \cite{HP}.
\begin{lemma}\label{HP-point}
Let $p> \frac{3n-2}{2n-1}$ and suppose that $\mu$ is finite signed measure in $\Om$.  If $u$ is a renormalized solution to 
\begin{eqnarray*}
\left\{ \begin{array}{rcl}
-{\rm div}\,\mathcal{A}(x, \nabla u)&=&  \mu  \quad \text{in} ~\Omega, \\
u&=&0  \quad \text{on}~ \partial \Omega,
\end{array}\right.
\end{eqnarray*}
and $\partial\Om$ is sufficiently flat,  then 
\begin{equation*}
|\nabla u(x)|\leq C \mathbf{T}[|\mu|](x)
\end{equation*}
for a.e. $x\in \mathbb{R}^n$, where $|\nabla u(x)|$ is set to be  zero outside $\Om$.  
\end{lemma}

We can now prove Theorem \ref{second-main-Ric}.

\begin{proof}[{\bf Proof of Theorem \ref{second-main-Ric}}] Let $\epsilon_0$ be as in the theorem and suppose that ${\rm supp}(\om)\subset\Om_{\delta_0}$.
	 In this proof, we shall fix a  $\kappa\in (0,\frac{\epsilon_0}{4n})$.
 
 By inequality (2.10) of \cite{Ph1} and condition \eqref{c0forom} we have 
  \begin{equation}\label{Z13}
  \mathbf{P}^R[\left(\mathbf{P}^R[|\omega|]\right)^q](x)\leq C [\om]^{\frac{q-p+1}{p-1}}_{M^{1, \frac{q}{q-p+1}}(\Om)}\mathbf{P}^{2R}[\omega](x)\leq C (c_0)^{\frac{q-p+1}{q}} \mathbf{P}^{R}[|\omega|](x), 
  \end{equation}
 for a.e. $x\in \Om$.

 Moreover, since ${\rm supp}(\om)\subset\Om_{\delta_0/2}$ we also have 
 \begin{equation*}
 \|\mathbf{P}^R[|\omega|]\|_{L^\infty(\mathbb{R}^n\backslash \Omega_{\delta_0/4})}\leq C(\delta_0,\beta,p,n, R) 	|\om|(\Om)^{\frac{1}{p-1}},
 \end{equation*}
 and thus by \eqref{Z1}, 
 \begin{align*}
 \mathbf{P}^R[\left\{d(\cdot)^{-\kappa}\mathbf{P}^R[|\omega|]\chi_{\Om}(\cdot)\right\}^q](x)\leq   C(\delta_0)	\mathbf{P}^R[\left(\mathbf{P}^R[\omega]\right)^q](x)+ C\,  |\om|(\Om)^{\frac{q}{(p-1)^2}},
 \end{align*}
 
 Combining this with \eqref{Z13} and condition \eqref{c0forom}, we find
 \begin{align*}
 \mathbf{P}^R[\left\{d(\cdot)^{-\kappa}\mathbf{P}^R[|\omega|]\chi_{\Om}(\cdot)\right\}^q](x)&\leq C\,   (c_0)^{\frac{q-p+1}{q}}	\mathbf{P}^R[|\om|](x)+  
 C \, |\om|(\Om)^{\frac{q-p+1}{(p-1)^2}} |\om|(\Om)^{\frac{1}{p-1}}\\
 &\leq C\,   (c_0)^{\frac{q-p+1}{q}}	\mathbf{P}^R[|\om|](x)+  
 C \, (c_0)^{\frac{q-p+1}{q(p-1)}}  |\om|(\Om)^{\frac{1}{p-1}}\\
 &\leq C [(c_0)^{\frac{q-p+1}{q}} +(c_0)^{\frac{q-p+1}{q(p-1)}}]	\mathbf{P}^R[|\om|](x)
 \end{align*}
 for a.e. $x\in\Om$. This gives 
 \begin{align}\label{ES1}
 \mathbf{T}[ \mathbf{T}[|\omega|]^q](x) \leq C [(c_0)^{\frac{q-p+1}{q}} +(c_0)^{\frac{q-p+1}{q(p-1)}}]	\mathbf{T}[|\om|](x)
\end{align}	
for a.e. $x\in\RR^n$.

\noindent {\bf Step 1:} In this step, we assume that  $\omega\in C_c^\infty$ with ${\rm supp}(\omega) \subset \Omega_{\delta_0}$.
Let us set
\begin{align*}
V=\left\{v\in W^{1,\, 1}_0:	 |\nabla v| \chi_\Om\leq N \mathbf{T}[|\omega|] \quad {\rm a.e.}\right\},
\end{align*}
where $N$ is to be determined. 
Since $\omega\in C^\infty_c(\Om)$, in view of  \eqref{Z2}  we have that
	\begin{equation*}
	|\nabla v(x)|^q\leq C(\omega)d(x)^{-q\kappa} \in L^{2n}(\Omega),
	\end{equation*}
 and in particular,   $|\nabla v|^q\in W^{-1,\frac{p}{p-1}}(\Om)$ for any $v\in V$.

	We next define a map $S: V\rightarrow W^{1,\, 1}_0$ by letting $S(v)=u$, where  $v\in V$   and $u$ is the unique renormalized solution to 
	\begin{eqnarray*}
		\left\{ \begin{array}{rcl}
			-{\rm div}\,\mathcal{A}(x, \nabla u)&=& |\nabla v|^q + \omega  \quad \text{in} ~\Omega, \\
			u&=&0  \quad \text{on}~ \partial \Omega.
		\end{array}\right.
	\end{eqnarray*}

By Lemma \ref{HP-point} and \eqref{ES1} we have 
\begin{align*}
|\nabla u| \chi_\Om &\leq {\bf T}[|\nabla v|^q\chi_\Om +|\om|]\\
& \leq C N^{\frac{q}{p-1}} \mathbf{T}[\mathbf{T}[|\om|]^q] + C \mathbf{T}[|\om|]\\
& \leq C N^{\frac{q}{p-1}} [(c_0)^{\frac{q-p+1}{q}} +(c_0)^{\frac{q-p+1}{q(p-1)}}] \mathbf{T}[|\om|] + C \mathbf{T}[|\om|]. 
\end{align*}

Thus if we choose $N=2C$ and $c_0$ sufficiently small we obtain that   $S(V)\subset V$. Moreover, using the results of \cite{DMOP}, it can be shown that 
$S$ is continuous and compact (see  also \cite{QH4}). Thus by Schauder Fixed Point Theorem,   there exists a solution  $u\in V$  to the equation \eqref{Riccati2}.

\bigskip

\noindent 	{\bf Step 2:} Let $\om_k =  \rho_k*\om$, where  $\{\rho_k\}_{k\in \mathbb{N}}$ is a standard sequence of mollifiers.  Choose $k$ sufficiently large so that $\om_k\in C^\infty_c(
\Om_{\delta_0/2})$ for all such $k$.  It is easy to see from condition \eqref{c0forom}
 that 
\begin{equation*}
[\om_k]^{\frac{q}{p-1}}_{M^{1, \frac{q}{q-p+1}}(\Om)}\leq A c_0,
\end{equation*}
where $A$ is independent of $k$.
Thus 	we may apply {\bf Step 1} with  $\omega=\om_k$ to obtain  a sequence of solutions $\{u_k\}\subset V$ to the equation  
	\begin{eqnarray*}
		\left\{ \begin{array}{rcl}
			-{\rm div}\,\mathcal{A}(x, \nabla u_k)&=& |\nabla u_k|^q + \om_k  \quad \text{in} ~\Omega, \\
			u_k&=&0  \quad \text{on}~ \partial \Omega.
		\end{array}\right.
	\end{eqnarray*}

Then we  apply the  results of \cite{DMOP} to get a subsequence $\{u_{k'}\}$  and function $u$ such that $\nabla u_{k'}\rightarrow \nabla u$  in $L^q(\Om)$ and $u$ is a renormalized solution of 
\eqref{Riccati2} (see also \cite{QH4}).		
\end{proof}

\end{document}